\documentclass[12pt]{article}

\usepackage{amstext    }
\usepackage{amsthm    }
\usepackage{a4}
\usepackage[mathscr]{eucal}
\usepackage{mathrsfs}

\usepackage{amsmath}
\usepackage{amssymb}
\usepackage{amscd}

\newtheorem{theorem}{Theorem}[section]

\newtheorem{proposition}[theorem]{Proposition}

\newtheorem{lemma}[theorem]{Lemma}
\newtheorem{remark}[theorem]{Remark}

\newtheorem{example}[theorem]{Example}

\newtheorem{problem}[theorem]{Problem}

\newcommand{\cali}[1]{\mathscr{#1}}

\newcommand{\volume}{{\rm vol}}
\newcommand{\lov}{{\rm lov}}
\newcommand{\supp}{{\rm supp}}

\newcommand{\dist}{{\rm dist}}
\newcommand{\diam}{{\rm diam}}

\newcommand{\chac}{{\rm \bf 1}}

\newcommand{\ddbar}{{\partial\overline\partial}}

\newcommand{\SO}{{\rm SO}}
\newcommand{\Homeo}{{\rm Homeo}}

\newcommand{\GLW}{{\rm GLW}}

\newcommand{\id}{{\rm id}}

\newcommand{\Cc}{\cali{C}}
\newcommand{\Dc}{\cali{D}}

\newcommand{\Fc}{\cali{F}}

\newcommand{\Lc}{\cali{L}}
\newcommand{\Mc}{\cali{M}}

\newcommand{\Pc}{\cali{P}}
\newcommand{\Qc}{\cali{Q}}

\newcommand{\Uc}{\cali{U}}

\newcommand{\C}{\mathbb{C}}
\newcommand{\D}{\mathbb{D}}

\newcommand{\N}{\mathbb{N}}
\newcommand{\Z}{\mathbb{Z}}
\newcommand{\R}{\mathbb{R}}
\newcommand{\T}{\mathbb{T}}
\newcommand{\B}{\mathbb{B}}
\newcommand{\U}{\mathbb{U}}

\renewcommand{\P}{\mathbb{P}}

\title{Entropy for hyperbolic Riemann surface laminations I}

\author{Tien-Cuong Dinh, Viet-Anh Nguy{\^e}n and Nessim Sibony}

\begin{document}

\maketitle

\begin{abstract}
We develop a notion  of entropy, using hyperbolic time, for laminations by hyperbolic Riemann surfaces.  When the lamination is compact and transversally smooth, we show that the entropy is finite and at least equal to 2. Moreover, the Poincar\'e metric on leaves is transversally H\"older continuous. A notion of metric entropy is also introduced for harmonic measures.  
\end{abstract}

\noindent
{\bf Classification AMS 2010:} 37F75, 37A.


\bigskip
\noindent
{\bf Notation.} Throughout the paper, $\D$ denotes the unit disc
in $\C$, $r\D$ denotes the disc of center 0 and of radius $r,$ and
$\D_R\subset\D$ is the disc of center $0$ and of radius $R$ with
respect to the Poincar{\'e} metric on $\D$,
i.e. $\D_R=r\D$ with $R:=\log[(1+r)/(1-r)]$. 
Poincar{\'e} metric on
a Riemann surface, in particular on $\D$ and on the leaves of a lamination, is
given by a positive $(1,1)$-form that we denote by $\omega_P$. The associated distance and diameter are denoted by $\dist_P$ and $\diam_P$. 
A leaf through a point $x$ of a lamination is often denoted by $L_x$ and $\phi_x:\D\to L_x$ denotes a universal covering map of $L_x$ such that $\phi_x(0)=x$.

\section{Introduction} \label{introduction}

The main goal of this   paper is to introduce  a notion of  entropy for possibly singular 
hyperbolic laminations  by Riemann surfaces. We also  study the  transverse 
regularity of the Poincar\'e metric and the finiteness of the entropy. In order to simplify the presentation, we will mostly  focus, in this first part, on compact laminations which are transversally smooth. 
We will study the case of singular foliations in the second part of this paper.

The question of hyperbolicity of leaves for generic foliations in $\P^k$ has
been adressed by many authors. We just mention here the case of
a polynomial vector field  in $\C^k.$ It induces
a foliation by Riemann surfaces  in the  complex projective space $\P^k.$ We
can consider that  this  foliation is   the  image of the
foliation in $\C^{k+1}$  given by a  holomorphic vector field
 $$F(z):=\sum_{j=0}^k  F_j(z) \frac{\partial}{\partial z_j} $$
with $F_j$ homogeneous polynomials of degree $d \geq 2$ without common factor.    

The  singular set
 corresponds  to
the union of the indeterminacy points of  $f= [F_0  : \cdots : F_k ]$ and
 the fixed
points of $f$ in $\P^k$.
The  nature of the leaves  as  abstract Riemann surfaces  has  received much
 attention.
Glutsyuk \cite{Glutsyuk} and Lins Neto  \cite{Neto}
 have  shown that   on  a generic  foliation $\Fc$ of degree $d$ the leaves
 are  covered  by the
unit disc in $\C$. We then  say  that the foliation is {\it hyperbolic}.  More precisely,
 Lins Neto  has shown that
this is the case  when all  singular points have  non-degenerate linear part. In
 \cite{CandelGomezMont} Candel-Gomez-Mont  have  shown
that  if all the  singularities are  hyperbolic, the Poincar\'e metric  on
leaves is transversally  continuous.
We will consider this situation in the second part of this work.

Let $(X,\Lc)$ be a (transversally) smooth compact lamination by hyperbolic Riemann surfaces.
We show in Section \ref{section_poincare} that the Poincar\'e metric on leaves is  transversally H\"{o}lder continuous.
The exponent of   H\"{o}lder continuity can be estimated in geometric terms.  The continuity was proved by Candel in \cite{Candel}.
The main tool of the proof is to use  Beltrami's equation in order to
compare universal covering maps of  any leaf $L_y$ near a given leaf $L_x$. More precisely, we first  construct a non-holomorphic parametrization $\psi$ from $\D_R$ to $L_y$ which is close to a universal covering map $\phi_x:\D\to L_x$. Precise geometric estimates on $\psi$ allow us to modify it, using Beltrami's equation. We then obtain a holomorphic map that we can explicitly compare with a universal covering map $\phi_y:\D\to L_y$.

Our second concern is to define  the  entropy of hyperbolic lamination possibly with singularities. A notion  of geometric entropy  for regular Riemannian foliations was  introduced by  Ghys-Langevin-Walczak \cite{GhysLangevinWalczak}, see also Candel-Conlon  \cite{CandelConlon1, CandelConlon2}  and Walczak  \cite{Walczak}.  It is
related to the entropy of the holonomy pseudogroup, which depends on the  chosen  generators. 
The basic idea
is to quantify how  much leaves get far apart transversally. The transverse regularity of the metric on leaves and the lack of singularities  play a role in the  finiteness of the entropy.  

Ghys-Langevin-Walczak show in particular that  when their geometric  entropy vanishes, the foliation admits a transverse measure. The survey by Hurder \cite{Hurder} gives an account on many important results in foliation theory and contains a large bibliography.

Our notion of entropy contains a large number of classical situations. An interesting fact is that this entropy is related to an increasing family of distances as in  Bowen's point of view \cite{Bowen}. This allows us for example to introduce other dynamical notions like metric entropy, local entropies or Lyapounov exponents. 

We first introduce in Section \ref{section_entropy} a general notion  of entropy on a metric space $(X,d).$  To a given family of distances 
$(\dist_t)_{t\geq 0}$, we associate an entropy which  measures the growth rate (when $t$ tends to infinity) of the number of balls of small radius $\epsilon$, in the  metric $\dist_t$,
needed in order to cover the space $X$. 

For hyperbolic  Riemann surface  laminations  we  define
$$
\dist_t(x,y):=\inf\limits_{\theta\in \R} \sup_{\xi\in \D_t} \dist_X(\phi_x(e^{i\theta}\xi),\phi_y(\xi)).
$$
Recall that $\phi_x$ and $\phi_y$ are universal covering maps for the leaves through $x$ and $y$ respectively with  $\phi_x(0)=x$ and  $\phi_y(0)=y$. These maps are unique up to a rotation on $\D$.  
The metric $\dist_t$ measures how  far two leaves
 get apart before  the hyperbolic time $t.$ It takes into account the time parametrization like in the classical
case where one measures the distance of two orbits before time $n$, by
measuring the distance at each time $i<n$. So, we  are not just concerned
 with geometric proximity.

We will show that our entropy is  finite  for compact hyperbolic laminations which are transversally smooth.
The notion of entropy can be extended to Riemannian foliations and
 a priori it is bigger than or equal to the geometric entropy introduced by  
 Ghys, Langevin and Walczak. 
 
 As for  the  tranverse regularity of the Poincar\'e metric, the main tool is to  estimate the distance  between leaves
 using the Beltrami equation in order to go from geometric  estimates to the analytic ones needed in our definition.
 The  advantage here is that the hyperbolic time  we choose is canonical. So, the value of the entropy  is  unchanged 
 under homeomorphisms  between laminations which are holomorphic along leaves. 
 
 The proof that the entropy is  finite  for singular 
 foliations is  quite delicate and requires  a  careful analysis  of the dynamics around the singularities.  
 We will consider this problem in the second part of the paper.
 We will discuss in Section \ref{section_entropy_metric} a notion of metric entropy for harmonic probability measures and give there some open questions.


\section{Poincar\'e metric on laminations} \label{section_poincare}

In this section, we give some basic properties of laminations by hyperbolic Riemann surfaces. We will show that the Poincar\'e metric on leaves of a smooth compact hyperbolic lamination is transversally H\"older continuous. 

Let $X$ be a locally compact space.  A {\it lamination} or {\it Riemannian lamination} $\Lc$ on $X$  is  the  data of  an atlas with charts 
$$\Phi_i:\U_i\rightarrow \B_i\times \T_i.$$
Here, $\T_i$ is a locally compact metric space, $\B_i$ is a domain in $\R^n$,
$\Phi_i$ is a homeomorphism defined on an open subset $\U_i$ of
$X$ and  all the changes of coordinates $\Phi_i\circ\Phi_j^{-1}$ are of the form
$$(x,t)\mapsto (x',t'), \quad x'=\Psi(x,t),\quad t'=\Lambda(t),$$
where $\Psi,\Lambda$ are continuous maps and $\Psi$ is smooth with
respect to $x$. 

The open set $\U_i$ is called a {\it flow
  box} and the manifold $\Phi_i^{-1}\{t=c\}$ in $\U_i$ with $c\in\T_i$ is a {\it
  plaque}. The property of the above coordinate changes insures that
the plaques in different flow boxes are compatible in the intersection of
the boxes. Two plaques are {\it adjacent} if they have non-empty intersection.
In what follows, we always reduce slightly  flow boxes in order to avoid a bad geometry near their boundaries. For simplicity we consider only $\B_i$ which are homeomorphic to a ball. 

A {\it leaf} $L$ is a minimal connected subset of $X$ such
that if $L$ intersects a plaque, it contains that plaque. So, a leaf $L$
is a connected real manifold of dimension $n$ immersed in $X$ which is a
union of plaques. It is not difficult to see that $\overline L$ is
also a lamination. A {\it chain of plaques} is a sequence $P_0,\ldots,P_m$ of plaques such that $P_i$ is adjacent to $P_{i+1}$ for $i=0,\ldots, m-1$. These plaques belong necessarily to the same leaf.  
 
A {\it transversal} in a flow box is a closed set of the box which intersects every
plaque in one point. In particular, $\Phi_i^{-1}(\{x\}\times \T_i)$ is a
transversal in $\U_i$ for any $x\in \B_i$.  In order to
simplify the notation, we often identify $\T_i$
with $\Phi_i^{-1}(\{x\}\times \T_i)$ for some $x\in \B_i$ or even
identify $\U_i$ with $\B_i\times\T_i$ via the map $\Phi_i$.

We are mostly interested in the case where the $\T_i$ are closed subsets of smooth real manifolds and the functions $\Psi,\Lambda$ are smooth in all variables. In this case, we say that the lamination is {\it smooth} or {\it transversally smooth}.
If, moreover, $X$ is compact, we can embed it in an $\R^N$ in order to use the distance induced by a Riemannian metric on $\R^N$. 
When $X$ is a Riemannian manifold and the leafs of $\Lc$ are manifolds immersed in $X$, we say that $(X,\Lc)$ is a {\it foliation} and we often assume that the foliation is {\it transversally smooth}, i.e., the maps $\Phi_i$ above are smooth. 

In the definition of laminations, 
 if  all the  $\B_i$   are  domains  in $\C$ and $\Psi$ is holomorphic with respect to $x$, we say that $(X,\Lc)$ is a
{\it Riemann surface lamination}.
Recall  that   a {\it Riemann surface lamination with singularities} is  the data
$(X,\Lc,E)$, where $X$ is a locally compact space, $E$ a closed
subset of $X$ and $(X\setminus E,\Lc)$ is a Riemann surface
lamination. The set $E$ is {\it the singularity set} of the lamination and  we   assume that $\overline{X\setminus E}=X$, see  e.g. \cite{DinhNguyenSibony2, FornaessSibony1}  for more details.

Consider now a smooth Riemann surface lamination $(X,\Lc)$. 
When we do not assume that $X$ is compact, our discussion can be applied to singular laminations by considering their regular parts.
Assume that the leaves of $X$ are all (Kobayashi) hyperbolic. Let $\phi_x:\D\to L_x$ be a universal covering map of the leaf through $x$ with $\phi_x(0)=x$. Then, the Poincar\'e metric on $\D$ induces a metric on $L_x$ which depends only on the leaf. The Poincar\'e metric on $L_x$ is given by a positive $(1,1)$-form that we always denote by $\omega_P$. The associated distance is denoted by $\dist_P$. 

Let $\omega$ be a Hermitian metric on the leaves which is transversally smooth. We can construct such a metric on flow boxes and glue them using a partition of unity.
We have
$$\omega=\eta^2\omega_P\quad \mbox{where}\quad 
\eta(x):=\|D\phi_x(0)\|.$$
Here, for the norm of the differential $D\phi_x$ we use the Poincar\'e metric on $\D$ and the Hermitian metric $\omega$ on $L_x$. 

The extremal property of the Poincar\'e metric implies that
$$\eta(x)=\sup\big\{\|D\phi(0)\|,\quad \phi:\D\to L \mbox{ holomorphic such that } \phi(0)=x\big\}.$$
Using a map sending $\D$ to a plaque, we see that the function $\eta$ is locally bounded from below by a strictly positive constant. When $X$ is compact and the leaves are hyperbolic, the classical Brody lemma implies that $\eta$ is also bounded from above.

Fix now  a distance $\dist_X$ on $X$ such that on flow boxes $\U=\B\times\T$ as above, it is locally equivalent to the distance induced by a Riemannian metric.
Here is the main theorem of this section.

\begin{theorem} \label{th_poincare_holder}
Let $(X,\Lc)$ be a smooth compact lamination by hyperbolic Riemann surfaces. Then the Poincar\'e metric on the leaves  is H\"older continuous, that is, the function $\eta$ defined above is H\"older continuous on $X$. 
\end{theorem}

The proof occupies the rest of this section. The result is also
valid for laminations which are transversally of class $\Cc^{2+\alpha}$ with
$\alpha>0$.
In order to simplify the notation, we embed $X$ in an $\R^N$ and use the distance $\dist_X$ induced by the Euclidean metric on $\R^N$.
Multiplying $\omega$ by a constant, we can assume that $\omega\leq \omega_P$ on leaves, i.e. $\eta\leq 1$. 
We also have $\omega_P\leq A\omega$, i.e.  $\eta\geq 1/A$,
 for some fixed constant $A\geq 1$. 
Fix also an atlas of $X$, fine enough. We will only consider finite atlases which are finer than this one. For simplicity, all the plaques we consider are small and simply connected. We also use a coordinate change on $\R^N$ and choose $A$ large enough such that $\dist_X\leq\dist_P\leq A\dist_X$ on plaques. The second inequality does not hold when we deal with singular foliations.

Let $\phi$ and $\phi'$ be two maps from a space $\Sigma$ to $X$. If $K$ is a subset of $\Sigma$, define 
$$\dist_K(\phi,\phi'):=\sup_{a\in K}\dist_X(\phi(a),\phi'(a)).$$
Consider constants $R\gg 1$ and $0<\delta\ll 1$ such that $e^{2R}\delta\leq 1$. 
We say that two points $x$ and $y$ in $X$ are {\it conformally $(R,\delta)$-close} if the following property is satisfied and if it also holds when we exchange $x$ and $y$. 

Let $\phi_x:\D\to L_x$ and $\phi_y:\D\to L_y$ be universal covering maps with $\phi_x(0)=x$ and $\phi_y(0)=y$. There is a smooth map $\psi:\overline \D_R\to L_y$ without critical point such that $\psi(0)=y$,
$\dist_{\overline \D_R}(\phi_x,\psi)\leq\delta$ and $\psi$ is $\delta$-conformal in the following sense. Since $\D_R$ is simply connected, there is a unique smooth map $\tau:\overline\D_R\to\D$ such that $\psi=\phi_y\circ\tau$ and $\tau(0)=0$. We assume that 
$\|D\tau\|_\infty\leq 2A$ and
{\it the Beltrami coefficient} $\mu_\tau$ of $\tau$ satisfies $\|\mu_\tau\|_{\Cc^1}\leq \delta$. 
Here, we consider the norm of the differential $D\tau$ with respect to the Poincar\'e metric on $\D$ and the norm of $\mu_\tau$ with respect to the Euclidean metric.
Recall also that $\mu_\tau$ is defined by
$${\partial\tau\over \partial\overline \xi} = \mu_\tau {\partial\tau\over \partial \xi}\cdot$$  

Note that the above notion is independent of the choice of $\phi_x$ and $\phi_y$ since these maps are defined uniquely up to a rotation on $\D$. We have the following important estimate.

\begin{proposition} \label{prop_quasi_close}
Let $x$ and $y$ be conformally $(R,\delta)$-close as above (in particular, $\delta\leq e^{-2R}$). There is a real number $\theta$ such that if $\phi_y'(\xi):=\phi_y(e^{i\theta}\xi)$, then
$$|\eta(x)-\eta(y)|\leq A' e^{-R} \quad \mbox{and}\quad 
\dist_{\D_{R/3}} (\phi_x,\phi_y')\leq A' e^{-R/3},$$
where $A'>0$ is a  constant independent of $R,\delta, x$ and $y$. 
\end{proposition}

In what follows, 
$\lesssim$ denotes an inequality up to a multiplicative constant independent of $R,\delta,x$ and $y$.
We will need the following quantitative version of Schwarz's lemma.

\begin{lemma} \label{lemma_iden_close}
Let $\widetilde\tau:\D_R\to \D$ be a holomorphic map such that $\widetilde\tau(0)=0$. Write $D\widetilde\tau(0)=\lambda e^{i\theta}$ with $\lambda>0$ and $\theta\in\R$. Assume that $1-\lambda\lesssim e^{-R}$. Then, we have
$$\dist_P(\widetilde\tau(\xi),e^{i\theta}\xi)\lesssim e^{-R/3} \quad \mbox{for}\quad  \xi\in\D_{R/3}.$$ 
\end{lemma}
\proof
We can assume that $\theta=0$. 
Since $R$ is large, we can compose $\widetilde\tau$ with a slight contraction in order to assume that $\widetilde\tau$ is defined from $\D$ to $\D$. The computation is essentially the same. So, we still have $1-\lambda\lesssim e^{-R}$ and $1-\lambda>0$, by Schwarz's lemma. 

Consider the holomorphic function $u:\D\to \D$ defined by
$$u(\xi):=\xi^{-1}\widetilde\tau(\xi) \quad \mbox{and} \quad u(0):=\lambda.$$ 
Observe that $1-|\xi|\gtrsim e^{-R/3}$ for $\xi\in \D_{R/3}$ and $\big|1-\overline{u(\xi)}|\xi|^2\big|\gtrsim 1-|\xi|$. Therefore, 
$$\dist_P(\widetilde\tau(\xi),\xi)=2\tanh^{-1}{|\xi||1-u(\xi)| \over \big|1-\overline{u(\xi)}|\xi|^2\big|}\lesssim e^{R/3}|1-u(\xi)|.$$
It is enough to show that $|1-u|\lesssim e^{-2R/3}$ on $\D_{R/3}$.
  
Since $u$ is holomorphic, it contracts the Poincar\'e metric on $\D$. So, it sends $\D_{R/3}$ to the disc of radius $R/3$ centered at $u(0)=\lambda$. We obtain the desired inequality using that  $\dist_P(0,\lambda)\geq R+o(R)$ because 
$0<1-\lambda \lesssim e^{-R}$. 
\endproof

\noindent
{\bf Proof of Proposition \ref{prop_quasi_close}.} 
We first construct a homeomorphism $\sigma:\D_R\to\D_R$, close to the identity, such that $\widetilde\tau:=\tau\circ \sigma^{-1}$ is holomorphic.  For this purpose,  it is enough to construct $\sigma$ satisfying the following Beltrami  equation
$${\partial\sigma\over\partial\overline\xi}=\mu_\tau {\partial\sigma\over\partial\xi}\cdot$$
Indeed, it is enough to compute the derivatives of $\tau=\widetilde\tau\circ\sigma$ and to use the above equation together with the property that $\|\mu_\tau\|_\infty<1$ in order to obtain that $\overline\partial \widetilde\tau=0$. 

It is well-known from the Ahlfors-Bers theory, see e.g. \cite[p.181]{EarleSchatz}, that there is a solution such that
$$\|\sigma-\id\|_{\Cc^1}\lesssim \|\mu_\tau\|_{\Cc^1}\lesssim\delta,$$
where we use the Euclidean metric on $\D_R$. We deduce that
$$\|\sigma^{-1}-\id\|_{\Cc^1}\lesssim\delta\lesssim e^{-2R}.$$
Moreover, we can also compose $\sigma$ with an automorphism of $\D_R$ in order to get that $\sigma(0)=0$.  Now, it is not difficult to see that 
$$\dist_P(\sigma^{-1}(\xi),\xi)=2\tanh^{-1}{|\sigma^{-1}(\xi)-\xi|\over \big|1-\overline{\sigma^{-1}(\xi)}\xi \big|}\lesssim e^R\delta\quad  \mbox{for}\quad  \xi\in\D_R.$$ 

Define $\widetilde\phi_y:=\phi_y\circ \widetilde\tau$. This is a holomorphic map. Recall that 
$\psi=\phi_y\circ \tau$ and $\|D\tau\|_\infty\leq 2A$. Therefore, since $\phi_y$ is isometric with respect to the Poincar\'e metric, we obtain from the previous estimates that
$$\dist_{\D_R}(\widetilde\phi_y,\psi)\lesssim  e^R\delta$$
which implies that
$$\dist_{\D_R}(\phi_x,\widetilde\phi_y)\lesssim e^R\delta.$$

For a constant $R_0>0$ small enough, $\phi_x$ and $\widetilde\phi_y$ send $\D_{R_0}$ to the same flow box where Cauchy's formula implies that 
$$\|D\phi_x(0)\|-\|D\widetilde\phi_y(0)\|\lesssim e^R\delta.$$
The extremal property of the Poincar\'e metric yields 
$$\|D\widetilde\phi_y(0)\|\leq {1\over r}\|D\phi_y(0)\|=(1+O(e^{-R}))\|D\phi_y(0)\|.$$
Recall that $\|D\phi_x(0)\|=\eta(x)$ and $\|D\phi_y(0)\|=\eta(y)$. We deduce that
$$\|D\phi_y(0)\|\geq \|D\widetilde\phi_y(0)\|+O(e^{-R}\eta(y)).$$
Therefore,
$$\eta(x)-\eta(y)\leq \|D\phi_x(0)\|-\|D\widetilde\phi_y(0)\| +O(e^{-R}\eta(y))\leq  e^R\delta+O(e^{-R}\eta(y)).$$
By symmetry we get 
$$|\eta(x)-\eta(y)|\lesssim  e^R\delta+  e^{-R}(\eta(x)+\eta(y)).$$
This, combined  with the hypothesis that  $e^{2R}\delta\leq 1$ and  ${1\over A}\leq  \eta\leq 1,$  implies  the first estimate in the proposition.

We deduce from the above inequalities that
$$\|D\phi_y(0)\| - \|D\widetilde\phi_y(0)\| \lesssim e^R\delta +e^{-R}(\eta(x)+\eta(y)).$$
Write $D\widetilde\tau(0)=\lambda e^{i\theta}$ with $\lambda\geq 0$ and $\theta\in\R$. 
Since  $e^{2R}\delta\leq 1$ and  ${1\over A}\leq  \eta\leq 1,$ and $\phi_y$ is isometric with respect to the Poincar\'e metric, we obtain that $1-\lambda \lesssim e^{-R}$. By Lemma \ref{lemma_iden_close}, we have
$$\dist_P(\widetilde\tau(\xi),e^{i\theta}\xi)\lesssim e^{-R/3} \quad \mbox{for}\quad  \xi\in\D_{R/3}.$$ 
Define $\phi_y'(\xi):=\phi_y(e^{i\theta}\xi)$. Since $\phi_y$ is isometric with respect to the Poincar\'e metric, we obtain that 
$$\dist_{\D_{R/3}} (\phi'_y,\widetilde\phi_y)\lesssim e^{-R/3}.$$
This, combined  with the above estimate on the distance between $\phi_x$ and $\widetilde\phi_y,$ implies the result. 
\hfill $\square$

\bigskip

We continue the proof of Theorem \ref{th_poincare_holder}. 
Fix a finite atlas $\Uc^l$ fine enough, another finer atlas $\Uc^n$ and a third one $\Uc^s$ which is finer than $\Uc^n$. Flow boxes and plaques of $\Uc^s$, $\Uc^n$ and $\Uc^l$ are said to be {\it small, normal and large} respectively.  
Moreover, we can construct these atlases so that the following property is true for a fixed constant $0<d\ll A^{-2}$ and for the distance
$\dist_X$ on plaques:
\begin{enumerate}
\item[(A1)] Any disc of diameter $d$ in a plaque is contained in a small plaque; small (resp. normal) plaques are of diameter less than $2d$ (resp. $10^4dA$); the intersection of any large plaque with any flow box is contained in a plaque of this box.
\end{enumerate}

Moreover, we can construct these atlases so that the following properties are satisfied. 
To each small flow box $\U^s$, we can associate a normal flow box $\U^n$ and a large flow box $\U^l$ such that $\U^s\Subset\U^n\Subset \U^l$ and for all plaques $P^s,P^n,P^l$ in $\U^s$, $\U^n$ and $\U^l$ respectively, the following holds:
\begin{enumerate}
\item[(A2)] If $P^n$ and $P^l$ are adjacent, then $P^n\subset P^l$ and $\dist_X(\partial P^l,P^n)\geq 10^6dA^2$;
\item[(A3)] If $P^s$ and $P^n$ are adjacent, then $P^s\subset P^n$ and $\dist_X(\partial P^n, P^s)\geq 10^2dA$;
\item[(A4)] The projection $\Phi$ from $P^n$ to $P^l$ is well-defined, smooth and its image is compact in $P^l$; the projection of $P^s$ in $P^l$ is compact in the projection of $P^n$.
\end{enumerate}
Here, we use $\dist_X$ in order to define the projection.
Fix a constant $\kappa > 1$ large enough such that
\begin{enumerate}
\item[(A5)] If $x$ is a  point in $P^n$, then for the $\Cc^2$-norm on $P^n$
$$\|\Phi-\id\|_{\Cc^2}\leq e^\kappa \dist_X(x,\Phi(x)).$$
\end{enumerate}

Note that in what follows, to each small flow box $\U^s$ we fix a choice of the associated boxes $\U^n$ and $\U^l$ and we will only consider projections from plaques to plaques as described above. Moreover, a small, normal or large plaque is associated to a unique small, normal or large flow box. 
We have the following property for some fixed constant $\epsilon_0>0$ small enough
\begin{enumerate}
\item[(A6)] Two points at distance less than $\epsilon_0$ belong to the same small flow box and 
$\dist_X(\partial \U^n,\U^s)>\epsilon_0$ for $\U^s, \U^n$ as above.
\end{enumerate}

Consider now two points $x, y$ such that $\dist_X(x,y)\leq e^{-10\kappa d^{-1}AR}$ for $R>0$ large enough. 
We will show that $|\eta(x)-\eta(y)|\lesssim e^{-R}$. This implies that 
$\eta$ is H\"older continuous with H\"older exponent $(10\kappa d^{-1}A)^{-1}$. 

Since $x$ and $y$ are close, by (A6), they belong to a small flow box $\U^s$. Consider the projection $x'$ of $x$ to the normal plaque containing $y$. We have $\dist_X(x,x')\leq e^{-10\kappa d^{-1}AR}$ and $\dist_X(x',y)\leq 2e^{-10\kappa d^{-1}AR}$. By Proposition \ref{prop_quasi_close}, it is enough to check that $x$ and $y$ are conformally $(R,\delta)$-close with $\delta:=e^{-2R}$. So, 
we have to construct the map $\psi$ satisfying the definition 
of conformally $(R,\delta)$-close points as above. 

We claim that it is enough to consider the case where $y=x'$.
Indeed, if we can construct a map $\psi$ for $x,x'$, there is a point $a\in\D$ such that $\dist_P(0,a)\simeq \dist_P(x',y)\lesssim e^{-10\kappa d^{-1}AR}$ and $\psi(a)=y$. There is an automorphism $u:\D_R\to\D_R$ very close to the identity such that $u(0)=a$. Therefore, if we are able to construct the map $\psi$ for $x$ and $x'$, we obtain such a map for $x,y$ by replacing $\psi$ by $\psi\circ u$. Since $u$ is very close to the identity, the estimates do not change much. So, we can assume that $y=x'$.

The map $\psi$ will be obtained by composing $\phi_x$ with  local projections
 from the leaf $L_x$ to the leaf $L_y$. The main problem is to show that the
map is well-defined.
Let $P^s_1$ be a small plaque containing $x$ and $\U_1^s$ the associated small flow box. It is clear that $y$ belong to the associated normal flow box $\U^n_1$. Denote by  
$Q^n_1$ the plaque of $\U_1^n$ containing  $y$. The 
projection $\Phi_1$ from  $P_1$ to $Q_1^n$ is well-defined as described above. 

Consider a chain $\Pc=\{P^s_1,\ldots,P^s_m\}$ of $m$ small plaques with $m\leq 3d^{-1}AR$.  
Denote by $\U_i^s$ the small flow box associated to $P^s_i$ and $\U_i^n, \U_i^l$ the normal and large flow boxes associated to $\U_i^s$. We have the following lemma.

\begin{lemma} \label{lemma_chain}
There is a unique chain
$\Qc=\{Q^n_1,\ldots, Q^n_m\}$ such that $Q_i^n$ is a plaque of $\U_i^n$. 
Moreover, the projection $\Phi_i$ from $P^s_i$ to $Q^n_i$ satisfies $\Phi_i=\Phi_{i+j}$ on $P^s_i\cap P^s_{i+j}$ for 
$0\leq j\leq 10A$. We also have $\dist_X(P^s_i,Q^n_i)\leq e^{-4\kappa R}$ for $i=1,\ldots,m$. 
\end{lemma}
\proof
Note that since $\dist_X(x,y)\leq e^{-10\kappa d^{-1}AR}$ and $m\leq 3d^{-1}AR$, the last assertion of the lemma is a consequence of the previous ones and the property (A5) applied to points in the intersections $P_i^s\cap P_{i+1}^s$.
Indeed, by induction on $i$, we obtain
$$\dist_X(P_i^s,Q_i^n)\leq e^{-10\kappa d^{-1}AR} e^{i\kappa}\leq e^{-4\kappa R}.$$

We prove the other assertions by induction on $m$. Assume these properties for $m-1$, i.e. we already have the existence and the uniqueness of $Q_i^n$ for $i<m$. We have to construct $Q_m^n$ and to prove its uniqueness.  

Let $Q^l_{m-1}$ be the large plaque associated to $Q^n_{m-1}$. If $Q_m^n$ exists, since it intersects $Q^n_{m-1}$, by (A1) and (A2), it is contained in $Q^l_{m-1}$ and then it is the intersection of $Q^l_{m-1}$ with $\U_m^n$. The uniqueness of $Q_m^n$ follows. 

Fix a point $z$ in $P^s_{m-1}\cap P^s_m$. Since $\Phi_{m-1}(z)$ is close to $z$, by (A6), $\Phi_{m-1}(z)$ belongs to $\U^n_m$. Define $Q_m^n$ as the plaque of $\U_m^n$ containing this point. So, $Q_m^n$ intersects $Q^n_{m-1}$ and $Q_0^n,\ldots ,Q_m^n$ is a chain. 
By (A4), $\Phi_{m-1}(z)$ is also the projection of $z$ to $Q^l_{m-1}$. Since $Q^n_m$ is contained in $Q^l_{m-1}$, necessarily, the projection $\Phi_m(z)$ of $z$ to $Q^n_m$ coincides with $\Phi_{m-1}(z)$.

Arguing in the same way, we obtain that $P_{i+j}^s$ is contained in $P_i^n$, $Q_{i+j}^n$ is contained in $Q_i^l$ when $j\leq 10A$ and then we obtain that $\Phi_i=\Phi_{i+j}$ on $P^s_i\cap P^s_{i+j}$.
\endproof 

\noindent
{\bf End of the proof of Theorem \ref{th_poincare_holder}.} 
We have to show
that $x$ and $y$ are conformally $(R, \delta)$-close for a map $\psi$ that we are going to
construct. We call also small plaque any open set in $\D$ which is sent bijectively by $\phi_x$ to a small plaque in $L_x$. Let $\gamma$ be a radius of $\D_R$. 
We divide $\gamma$ into equal intervals $\gamma_i$ of Poincar\'e length $\simeq d/2$. By (A1), since $\dist_X\leq\dist_P\leq A\dist_X$ on plaques, we can find a small plaque $P_i$ containing $\gamma_i$ such that $\dist_P(\gamma_i,\partial P_i)\geq d/(4A)$. So, we have a chain of $m\simeq 2d^{-1}R$ plaques which covers $\gamma$. 
Define $P^s_i:=\phi_x(P_i)$, $Q_i^n$ as in Lemma \ref{lemma_chain} and $\psi:=\Phi_i\circ \phi_x$. We will check later that $\psi$ is well-defined on $\D_R$.

It follows from the last assertion in Lemma \ref{lemma_chain} that $\dist_{\D_R}(\phi_x,\psi)\leq e^{-4\kappa R}\leq\delta$. 
We also deduce from (A5) that $\psi$ has no critical point and that its Beltrami coefficient satisfies 
$\|\mu\|_{\Cc^1}\lesssim e^{-4\kappa R}$ for the Poincar\'e metric on $\D$ and then 
$\|\mu\|_{\Cc^1}\leq e^{-2\kappa R}\leq \delta$ for the Euclidean metric on $\D_R$. Here, we use that $\phi_x$ and $\phi_y$ are isometries with respect to the Poincar\'e metric.  

The property $\| D\tau\|_\infty\leq 2A$ is also clear since we have $\|D\Phi_i\|_\infty\leq 2$, locally
$\tau=\phi_y^{-1}\circ \psi=\phi_y^{-1}\circ\Phi_i\circ \phi_x$ and $\dist_X\leq \dist_P\leq A\dist_X$ on plaques.
This implies that $x,y$ are conformally $(R,\delta)$-close. It remains to check that $\psi$ is well-defined.

If $P_i\cap P_{i+j}\not=\varnothing$, then by (A1), we have $\diam_P(P_i\cup P_{i+j})\leq 4dA$ because $\dist_P\leq A\dist_X$. Hence, $j\leq 10A$. By Lemma \ref{lemma_chain}, $\psi$ is well-defined on $P_i\cup P_{i+j}$. So, $\psi$ is well-defined on the union $W$ of $P_i$ and this union contains the radius $\gamma$.  We will show later that $\psi$ extends 
to the union $W'$ of all small plaques which intersect $\gamma$. 
Of course, we only use projections from plaques to plaques in order to define the extension of $\psi$. So, the extension is unique. Let $\gamma'$ be another radius of $\D_R$ such that the angle between $\gamma$ and $\gamma'$ is small enough, e.g. less than $d/(10^2Ae^{2R})$. Then, $\gamma'$ is contained in $W'$. Observe that if we repeat the same construction of $\psi$ for $\gamma'$, the plaques $P_i'$ used to cover $\gamma'$ intersect  $\gamma$ because $\dist_P(\partial P_i',\gamma_i')\geq d/(4A)$. Therefore, the obtained values of $\psi$ on $\gamma'$ coincide with the above extension to $W'$. A simple compactness argument implies the existence of a well-defined map $\psi$ on $\D_R$. 

We check now that $\psi$ can be extended from $W$ to $W'$. 
Consider a plaque $P$ which intersects $\gamma_i$ and a plaque $\widetilde P$ which intersects $\gamma_{i+j}$. Assume that $P\cap \widetilde P\not=\varnothing$. It suffices to check that $\psi$ can be extended  to $W\cup P\cup \widetilde P$. As above, we obtain that $j\leq 10A$. This allows us by the previous arguments to see that $P^s:=\phi_x(P)$, $\widetilde P^s:=\phi_x(\widetilde P)$ and $P^s_i,\ldots P^s_{i+j}$ belong to the normal plaque $P^n_i$. The projection from $P^n_i$ to $Q^l_i$ gives us the unique extension of $\psi$. The proof of Theorem \ref{th_poincare_holder} is now complete.
\hfill $\square$

\section{Hyperbolic entropy for foliations} \label{section_entropy}

In this section, we introduce a general notion of entropy, which permits to
 describe some natural situations in dynamics and in foliation theory. We will show that the entropy of any compact smooth lamination by hyperbolic Riemann surfaces is finite.

Let $X$ be a metric space endowed with a distance $\dist_X$. Consider a family $\Dc=\{\dist_t\}$ of distances on $X$ indexed by $t\in \R^+$. We can also replace $\R^+$ by $\N$ and in practice we often have that $\dist_0=\dist_X$ and that 
$\dist_t$ is  increasing  with respect to  $t\geq 0$. 
In several  interesting situations the metrics $\dist_t$ are continuous with respect to $\dist_X.$

Let $Y$ be a non-empty subset of $X$. Denote by $N(Y,t,\epsilon)$ the minimal number of balls of radius $\epsilon$ with respect to the distance $\dist_t$ needed to cover $Y.$ Define the {\it  entropy} of $Y$ with respect to $\Dc$ by
$$h_\Dc(Y):=\sup_{\epsilon>0} \limsup_{t\rightarrow\infty} {1\over t} \log  N(Y,t,\epsilon).$$
When $Y=X$ we will denote by $h_\Dc$ this entropy. When $X$ is not compact, we can also consider the supremum of the entropies on compact subsets of $X$. Note that if $Y$ and  $Y'$ are two subsets of $X$, then 
$h_\Dc(Y\cup Y')=\max(h_\Dc(Y),h_\Dc(Y'))$.

Observe that when $\dist_t$ is increasing,  $N(Y,t,\epsilon)$ is  increasing  with respect to  $t\geq 0$. Moreover,   
$$ \limsup_{t\rightarrow\infty} {1\over t} \log  N(Y,t,\epsilon)$$  is  increasing
when $\epsilon$  decreases. So, in the above definition, we can replace $\sup_{\epsilon>0}$ by $\lim_{\epsilon\to 0^+}$. 
If $\Dc=\{\dist_t\}$  and $\Dc'=\{\dist'_t\}$ are two families of distances on $X$  such that  $\dist'_t\geq  A\dist_t$ for all $t$ with  a fixed constant $A>0$,  then  $h_{\Dc'}\geq h_\Dc$.

A subset $F\subset Y$ is said to be {\it $(t,\epsilon)$-dense in $Y$} if the balls of radius $\epsilon$ with respect to $\dist_t$, centered at a point in $F$, cover $Y$. Let $N'(Y,t,\epsilon)$ denote the minimal number of points in a $(t,\epsilon)$-dense subset of $Y$. 

Two points $x$ and $y$ in $X$ are said to be {\it $(t,\epsilon)$-close} if $\dist_t(x,y)\leq\epsilon$. 
A subset $F\subset X$ is said to be {\it $(t,\epsilon)$-separated} if for all distinct points $x,y$ in $F$ we have $\dist_t(x,y)>\epsilon$. Let $M(Y,t,\epsilon)$ denote the maximal number of points in a $(t,\epsilon)$-separated family $F\subset Y$. 
The  proof of the  following  proposition is  immediate.

\begin{proposition}\label{prop_comparison} 
We have
$$N(Y,t,\epsilon)\leq N'(Y,t,\epsilon)\leq M(Y,t,\epsilon)\leq N(Y,t,\epsilon/2).$$
In particular, 
$$h_\Dc(Y)=\sup_{\epsilon>0} \limsup_{t\rightarrow\infty} \frac{1}{ t} \log  N'(Y,t,\epsilon)=\sup_{\epsilon>0} \limsup_{t\rightarrow\infty} \frac{1}{ t} \log  M(Y,t,\epsilon).$$
\end{proposition}

 The following proposition gives a simple criterion for the finiteness of entropy.
We will see that  this  criterion applies for smooth laminations by Riemann surfaces.
 
 \begin{proposition} \label{prop_abstract_finite_entropy}
 Assume that there are positive constants $A$ and $m$ such that
for every $\epsilon>0$ small enough $X$ admits a covering by less than $A\epsilon^{-m}$ 
balls of radius $\epsilon$ for the distance $\dist_X$. Assume also that  
$$\dist_t\leq e^{ct+d}\dist_X+\varphi(t)$$
for some constants $c,d\geq 0$ and  a function $\varphi$  with $\varphi(t)\to 0$ as $t\to\infty$. Then, the entropy $h_\Dc$ is at most equal to $mc$.
\end{proposition}
\proof
Fix a constant $\epsilon$ small enough and consider only $t$ large enough so that $\varphi(t)\leq \epsilon/2$. If $x$ and $y$ are $\epsilon$-separated for $\dist_t$, then they are ${1\over 2}e^{-ct-d}\epsilon$-separated for $\dist_X$. 
In particular, they cannot belong to a same ball of radius ${1\over 4}e^{-ct-d}\epsilon$ with respect to $\dist_X$. Therefore,
it follows from the hypothesis that  
$$M(X,t,\epsilon)\leq A4^m e^{mct+md} \epsilon^{-m}.$$ 
We easily deduce that $h_\Dc\leq mc$. 
\endproof

Consider now a general dynamical situation. We call {\it time space} a data $(\Sigma,\dist_\Sigma,0_\Sigma,G)$, where 
$(\Sigma,\dist_\Sigma)$ is a metric space, $0_\Sigma$ is a point of $\Sigma$ that we call {\it time zero} and $G$ a group of isometries of $\Sigma$ with $0_\Sigma$ as a common fixed point. The elements of $G$ are called {\it time reparametrizations}. 

In practice, the metric $\dist_\Sigma$ is complete, $G$ is either $\{\id\}$ or the group of all the isometries fixing $0_\Sigma$ and preserving the orientation of $\Sigma$. The space $\Sigma$ can be 
\begin{enumerate}
\item[(G1)] one of the sets $\N$, $\Z$, $\R^+$, $\R$, $\C,$ $\R^p,$  $\C^p$ endowed with the usual distance;
\item[(G2)]  or a group with a finite system of generators  stable  under  inversion;
\item[(G3)] or the unit disc $\D$ in $\C$ with the Poincar{\'e} metric.
\end{enumerate}
For laminations by Riemann surfaces, we will consider essentially the last case
where $G$ the group of
rotations around $0\in\D$. Note that the case where $\Sigma$ is another symmetric domain may be also of interest.

Let $(X,\dist_X)$ be a metric space as above. Define for a subset $K\subset \Sigma$ and
two maps $\phi$, $\phi'$ from $\Sigma$ to $X$
$$\dist^G_K(\phi,\phi'):= \inf_{\sigma,\sigma'\in
  G}\ \sup_{s\in K}\ \dist_X\big(\phi\circ\sigma (s),\phi'\circ \sigma'(s)\big).$$
When $G=\{\id\}$, we have
$$\dist^G_K(\phi,\phi')=\dist_K(\phi,\phi'):=\sup_{s\in K} \dist_X\big(\phi(s),\phi'(s)\big).$$
For $t>0$ let
$$\Sigma_t:=\big\{s\in\Sigma,\ \dist_\Sigma(0_\Sigma,s)< t\big\}.$$
This set is invariant under the action of $G$. We define
$$\dist_{t}(\phi,\phi'):=\dist^G_{\Sigma_t}(\phi,\phi')=
\inf_{\sigma\in  G}\ \sup_{s\in \Sigma_t}\ \dist_X\big(\phi\circ\sigma (s),\phi'(s)\big).$$

Consider now a family $\Mc$ of maps from $\Sigma$ to $X$ satisfying the following properties:
\begin{enumerate}
\item[(M1)] for every $x\in X$ there is a map $\phi\in\Mc$ such that $\phi(0_\Sigma)=x$;
\item[(M2)] if $\phi,\phi'$ are two maps in $\Mc$ such that $\phi(0_\Sigma)=\phi'(0_\Sigma)$,
then $\phi=\phi'\circ\tau$ for some $\tau\in G$.
\end{enumerate}
So, $X$  is  "laminated" by images of $\phi\in \Mc:$ 
for every $x\in X$ there is a unique (up to time reparametrizations) 
map $\phi\in\Mc$ which sends the time zero $0_\Sigma$ to $x$. 
 We  get then  a natural family  $\{\dist_t\}_{t\geq 0}$ on $X.$

Define for $x$ and $y$ in $X$
$$\dist_t(x,y):=\dist_{t}(\phi_x,\phi_y),$$
where $\phi_x$, $\phi_y$ are in $\Mc$ such that $\phi_x(0_\Sigma)=x$ and
$\phi_y(0_\Sigma)=y$. The definition is independent of the choice of $\phi_x$ and $\phi_y$. It is clear that $\dist_0=\dist_X$ and that the family $\Dc:=\{\dist_t\}_{t\geq 0}$ is increasing when $t$ increases. For $Y\subset X$, denote by  
$h(Y)$ the associated entropy, where 
we drop the index $\Dc$ for simplicity. 

Observe  that  the entropy depends on  the metrics on $\Sigma$ and on $X$. Nevertheless, the entropy does not change if we modify $\dist_X$ on a compact set keeping the same topology or if we replace $\dist_X$ by another distance $\dist_X'$ such that $A^{-1}\dist_X\leq\dist_X'\leq A\dist_X$ for some constant $A>0$. 

We review some classical situations where we assume that $G=\{\id\}$.

\begin{example} \label{ex_iterate}
\rm
Consider a   continuous map $f:X\rightarrow X$ and $f^n:=f\circ\cdots\circ
f$ ($n$ times) the iterate of order $n$ of $f$. For $x\in X$ define a map
$\phi_x:\N\rightarrow X$ by $\phi_x(n):=f^n(x)$. For $\Sigma=\N$ and
$\Mc$ the family of these maps $\phi_x$, we obtain the {\it topological entropy} of $f$. More precisely, two points $x$ and $y$ in $X$ are
$(n,\epsilon)$-separated if 
$$\dist_n(x,y):=\max\limits_{0\leq  i\leq  n-1} \dist_X(f^i(x),f^i(y))>\epsilon.$$ 
If $f$ is  $K$-Lipschitz  continuous then $ \dist_n\leq K^n\dist_X.$

A subset $F$ of $X$ is
$(n,\epsilon)$-separated if its points are mutually
$(n,\epsilon)$-separated and the topological entropy of $f$ is given
by the formula
$$h(f):=\sup_{\epsilon>0}\limsup_{n\rightarrow\infty}{1\over
  n} \log\sup\big\{\# F,\ F\subset X \quad
(n,\epsilon)\mbox{-separated}\big\}.$$
This  notion was  introduced by Adler, Konheim and Mc Andrew.
The above formulation  when  $f$ is  uniformly continuous  was introduced by Bowen \cite{Bowen}, see  also Walters   \cite{Walters} or Katok-Hasselblatt  \cite{KatokHasselblatt}. 
When $f$  is  only continuous,  Bowen considers the entropy of compact sets  and then takes the  supremum over 
compacta.  

When  $f$  is a  meromorphic  map  on a compact  K\"{a}hler manifold $M$
with indeterminacy  set $I$,  we  can define $X:=M\setminus \bigcup_n(f^{-1})^n(I).$ Then  $f$ is  in general not uniformly continuous
on $X.$ However, it is  shown in \cite{DinhSibony}  that  the  entropy of $f|_X$  as  defined  above is   finite. 
Indeed, it is  dominated by the logarithm of the maximum of dynamical  degrees, see also Gromov \cite{Gromov} and Yomdin \cite{Yomdin} when $f$ is holomorphic or smooth. 

We can define the entropy of a map in  a
more general context.  Suppose  that $f$ is  only defined in $U\subset X.$ We  define 
$\dist_n(x,y)$ only when  $f^j(x)$ and $f^j(y)$   are well-defined for $j< n.$
Let $U_\infty:=\cap f^{-n}(U).$ All the  $\dist_n(\cdot,\cdot)$ are well-defined on $U_\infty\times U_\infty$ and we can consider the entropy of $f$ on $U_\infty.$ This situation occurs naturally in holomorphic dynamics.  See, e.g.
the case of  horizontal-like maps \cite{DinhNguyenSibony1}.  Note that the case where $\Sigma$ is a tree is also interesting because it allows to consider the dynamics of correspondences. 
\end{example}

\begin{example}\label{ex_flow}\rm
Consider a flow $(\Phi_t)_{t\in\R}$ on a compact Riemannian manifold
$X$. Define for $x\in X$ a map $\phi_x:\R^+\rightarrow X$ by
$\phi_x(t):=\Phi_t(x)$.  If $\Sigma=\R^+$ and $\Mc$  is the family 
$\{\phi_x\}$,   then the distance  $\dist_t$ with $t\geq 0$ is given by
$$\dist_t(x,y):=\sup_{0\leq  s < t} \dist(\Phi_{s}(x),\Phi_{s}(y)).$$ 
We obtain the classical entropy of the flow
$(\Phi_t)_{t\in\R}$ which is also equal to the topological entropy $h(\Phi_1)$ of $\Phi_1$, see
e.g. Katok-Hasselblatt  \cite[Chapter 3, p.112]{KatokHasselblatt}.
This notion can be extended without difficulty to complex flows. 
\end{example}

\begin{example}\rm
(see Candel-Conlon \cite{CandelConlon1, CandelConlon2}  and Walczak \cite{Walczak})   Let $\Gamma$ be a group with a finite system of generators $A$. We
assume that if $g\in A$ then $g^{-1}\in A$.  The distance $\dist_\Gamma$ between $g$
and $g'$ in $\Gamma$ is the minimal number $n$ such that we can
write $g^{-1}g'$  as a composition of $n$ elements in $A$.
The neutral element $\chac_\Gamma$ is considered as the
origin. Consider an action of $\Gamma$ on the left of a metric space
$X$, that is, a representation of $\Gamma$ in the group of bijections
from $X$ to $X$. Define for $x\in X$ the map $\phi_x:\Gamma\rightarrow X$ by
$\phi_x(g):=gx$. For $\Sigma:=\Gamma$ and $\Mc$ the family 
$\{ \phi_x\}$, we obtain the entropy of the action of $\Gamma$ on $X$.   More  precisely,  let $\Gamma_n$  be the ball  of center $\chac_\Gamma$  and  radius $n$  in $\Gamma$
with respect to the metric  introduced above.  Then
$$\dist_n(x,y):=\sup\limits_{g\in\Gamma_n}\dist_X(gx,gy).$$

The entropy depends on the metric on $\Gamma$,
i.e.  on the choice  of the system of generators $A$. We will denote it by $h_A$.
If $A'$ is another system of generators, there is a constant $c\geq 1$ independent of the action of $\Gamma$ on $X$ such that
$$c^{-1}h_{A'}\leq h_A\leq c h_{A'}.$$

The  function describing  the growth of $\Gamma$ is
$$
\lov_A(\Gamma):=\limsup_{n\to\infty}  {\log {\# \Gamma_n}\over n}\cdot
$$
It  also depends on the choice of generators.  If the map $x\mapsto gx$ is uniformly Lipschitz for each generator $g$,  
we can compare
$\lov_A(\Gamma)$ and $h_A.$ We get that if $X$ has  a finite box measure then
$h_A\leq    c\cdot \lov_A(\Gamma)$
for some positive constant $c$.

When $\Gamma$ is a hyperbolic group in the sense of Gromov \cite{Gromov2}, its Cayley graph can be compactified and the action of $\Gamma$ extends to the boundary $X$ of the Cayley graph. This allows to define a natural notion of entropy for $\Gamma$ which depends on the choice of generators.

The notion of entropy can be extended to any semi-group endowed with an invariant
distance and then covers Examples \ref{ex_iterate} and  \ref{ex_flow}. 
\end{example}

We now consider the case of laminations, where the group $G$ of time re-parametrization is not trivial.

\begin{example}\rm 
Let $(X,\Lc)$ be a compact Riemannian laminations  without singularities in a Riemannian manifold $M$. Assume that the lamination is transversally smooth.
 Ghys, Langevin and Walczak   \cite{GhysLangevinWalczak}    introduced and
studied a notion of geometric entropy $h_{\rm GLW}$. It can be summarized as follows. Define 
that $x$ and  $y$   are 
$(R,\epsilon)$-separated if $\delta_R(x,y)>\epsilon$ where $\delta_R$ is defined  below.   

Denote by $\exp_x:\R^n\to L_x$ the exponential map for $L_x$ such that $\exp_x(0)=x$. Here, we identify the tangent space of $L_x$ at $x$ with $\R^n$. So, $\exp_x$ is defined uniquely up to an element of the group $\SO(n)$. Define 
$$\delta'_R(x,y):=\inf_h  \sup_{\xi\in B_R} \dist_X (\exp_x(\xi),h(\xi)) \quad \mbox{and} \quad
\delta_R(x,y):=\delta_R'(x,y)+\delta_R'(y,x).$$
Here, $B_R$ denotes the ball of radius $R$ in $\R^n$ and $h:B_R\to L_y$ is a continuous map
with $h(0)=y$. This function $\delta_R$ measures the spreading of leaves. It seems that in general $\delta_R$ does not satisfy the triangle inequality. 
We have
$$h_\GLW:=\sup_{\epsilon>0} \limsup_{R\to \infty} {1\over R} \log\big\{ \#F,\quad F\subset X\quad (R,\epsilon)\mbox{-separated as above}\big\}.$$

In our approach, choose $\Sigma=\R^n$, $G=\SO(n)$ and $\Mc$ the family of all the exponential maps considered above. This allows us to define an entropy $h(\Lc)$. 
Indeed, we define
$$d_R(x,y):=\inf _{g\in\SO(n)} \sup_{\xi\in B_R}\dist_X\big(\exp_x(g(\xi)),\exp_y(\xi)\big).$$

It is not difficult to see that
$$  h_{\GLW}(\Lc)\leq h(\Lc) .$$
\end{example}

In the rest of this section, we consider a Riemann surface lamination $(X,\Lc)$. We assume that all its leaves are hyperbolic. Choose $(\Sigma,0_\Sigma)=(\D,0)$ endowed with the Poincar\'e metric. 
The group $G$ is the family of all rotations around 0. Define $\Mc$ as the family of all the universal covering maps $\phi:\D\to L$ associated to a leaf $L$. We obtain from the abstract formalism  
an entropy that we denote by $h(\Lc)$. We call it {\it the hyperbolic entropy} of the lamination. 

\begin{example} \rm \label{ex_disc}
Consider the case where $X$ is the Poincar\'e disc or a compact hyperbolic Riemann surface endowed with the Poincar\'e metric. Fix a constant $\epsilon>0$ small enough. Lemma \ref{lemma_poincare_separated} below shows that the property that $x,y$ are $(R,\epsilon)$-separated is almost equivalent to the property that $\dist_X(x,y)\simeq e^{-R}$. It follows that the entropy of a compact subset of $X$ is equal to its box dimension. 

When $X$ is a compact smooth lamination, we can choose a metric on $X$ which is equivalent to the Poincar\'e metric on leaves. We see that moving along the leaves contributes 2 to the entropy of the lamination. This property is new in comparison with the theory of iteration of maps, but it is not verified for general non-compact laminations.  
\end{example}

\begin{lemma} \label{lemma_poincare_separated}
Let $0<\epsilon<1$ be a fixed constant small enough. Then, there exist a constant $A\geq 1$ satisfying the following properties for all points $a$ and $b$ in $\D$.
\begin{enumerate}
\item[(i)] If $\dist_P(a,b)\leq A^{-1} e^{-R}$, then there are two automorphisms
$\tau_a,\tau_b$ of $\D$ such that $\tau_a(0)=a$, $\tau_b(0)=b$ and $\dist_P(\tau_a(\xi),\tau_b(\xi))\leq \epsilon$ for every $\xi\in \D_R$. 
\item[(ii)] If $Ae^{-R}\leq \dist_P(a,b)\leq 1$, then for all automorphisms
$\tau_a,\tau_b$ of $\D$ such that $\tau_a(0)=a$ and $\tau_b(0)=b$, we have $\dist_P(\tau_a(\xi),\tau_b(\xi))>\epsilon$ for some $\xi\in\D_R$. 
\end{enumerate}
\end{lemma}
\proof
(i) Since we use here invariant metrics, it is enough to consider the case where $b=0$ and $0\leq a<1$. 
We can also assume that $|a|\leq A^{-1}e^{-R}$, where $A\geq 1$ is a fixed large constant depending on $\epsilon$.  Consider the automorphisms $\tau_b:=\id$ and $\tau_a:=\tau$ with
$$\tau(z):={z+a\over 1+az}\cdot$$
We compare them on $\D_R=r\D$, where $e^R=(1+r)/(1-r)$. 

If $r$ is not close to 1, the Poincar\'e metric is comparable with the Euclidean metric on $\D_R$ and on $\tau(\D_R)$ and it is not difficult to see that $\tau$ is close to $\id$ on $\D_R$. So, we can assume that $r$ is close to 1. 

We have $|a|\ll 1-r$ and for $|z|=|x+iy|\leq r$
$$\dist_P(z,\tau(z))=2\tanh^{-1}{|\tau(z)-z|\over |1-\overline z\tau(z)|}\simeq 2\tanh^{-1}{a|1-z^2|\over \sqrt{(1-|z|^2)^2+4a^2y^2}}\cdot$$
Since $|ay|\leq |a| \lesssim 1-|z|$ and $1-|z|^2\simeq 1-|z|$, the last expression is of order
$$\tanh^{-1}{a|1-z^2|\over 1-|z|}\ll 1 \quad \mbox{when}\quad |z|\leq r.$$ 
This give the first assertion of the lemma.

\medskip

(ii) As above, we can assume that $b=0$.
We can replace $\tau_a$ by $\tau_a\circ \tau_b^{-1}$ in order to assume that $\tau_b=\id$. Fix a constant $A>0$ large enough depending on $\epsilon$ and assume that $Ae^{-R}\leq \dist_P(0,a)\leq 1$.
So, $R$ is necessarily large and $r$ is close to 1. We first consider the case where $\tau_a=\tau$. For $z=ir$, the above computation gives
$$\dist_P(\tau(z),z)\simeq 2\tanh^{-1}{2a\over \sqrt{4(1-r^2)+4a^2}}\cdot$$
This implies that $\dist_P(\tau(z),z)> 4\epsilon$ if $A$ is large enough.

Consider now the general case where $\tau_a$ differs from $\tau$ by a rotation. There is a constant $-\pi\leq\theta\leq\pi$ such that $\tau_a(z)=\tau(e^{i\theta}z)$. Without loss of generality, we can assume that $-\pi\leq \theta\leq 0$. It is enough to show that  
$$\dist_P(\tau_a(w),w)=\dist_P(\tau(z),w)>\epsilon$$ 
for $z=ir$ and $w=e^{-i\theta}z$.
 
Observe that since $\tau$ is conformal and fixes $\pm 1$, it preserves the circle arc 
through $-1,1$ and a point in $i\R$. Moreover, we easily check that the real part of $\tau(z)$ is positive and that $\tau(z)$ is outside $\D_R$. 
So, the geodesic joining $z_1:=\tau(z)$ and its projection $z_2$ (with respect to the Poincar\'e metric) to $i\R$ intersects $\partial \D_R$ at a point $z_3$. 
If $\dist_P(z_1,z_2)>\epsilon$, since $w$ is on the left side of $i\R$, we have necessarily $\dist_P(z_1,w)>\epsilon$. 
Otherwise, we have $\dist_P(z_1,z_3)\leq\epsilon$ and since $\dist_P(z,z_1)>4\epsilon$, we deduce that 
$\dist_P(z,z_3)>3\epsilon$. Now, since $w,z,z_3$ are in $\partial\D_R$, we see, using   an automorphism which sends $z_3$ to 0, that $w$ is further than $z$ from $z_3$, i.e. $\dist_P(w,z_3)>\dist_P(z,z_3)$. It follows that $\dist_P(w,z_3)>3\epsilon$ and hence $\dist_P(w,z_1)>\epsilon$. This completes the proof of the lemma.
 \endproof

\begin{example}\rm
Let $S$ be  a hyperbolic compact Riemann surface. Let $\Gamma$ denote the group of deck transformations of $S$, i.e. $\Gamma:\simeq\pi_1(S)$.   Assume also that $\Gamma$ acts on a compact metric space $N$ as a group of homeomorphisms, that is, we have a representation of $\Gamma$ into $\Homeo(N)$. 
For example, we can take $N=\partial\D$ or $\P^1$.
Consider now the suspension which gives us a lamination by Riemann surfaces. 
More precisely,  let $\widetilde{S}\simeq\D$  be the universal  covering  of $S.$ The group $\Gamma$ acts on 
$\widetilde S\times N$ by homeomorphisms
$$(\tilde{s},x)\mapsto \big(\gamma \tilde{s},\gamma x\big)\quad \mbox{with}\quad \gamma\in\Gamma.$$  This action is proper and discontinuous. The quotient $X:=\Gamma\backslash (\widetilde{S}\times N)$ is compact and
has  a  natural structure of a lamination by Riemann surfaces. Its leaves are the images of $\widetilde{S}\times \{x\}$ under the canonical projection $\pi:\  \widetilde{S}\times N\to X.$ 

Observe that the entropy of this lamination depends only on the representation of $\Gamma$ in $\Homeo(N)$. So, we call it the entropy of the representation. It would be interesting to study these quantities as functions on moduli spaces of representations. 

The  entropy of the group $\Gamma$ with respect to a system of generators is comparable with the entropy of the lamination. In particular, we  can have  laminations with positive  entropy and with a transverse measure. More precisely, let $f$ be a homeomorphism with positive entropy on a compact manifold $N$. It induces an action of $\Z$ on $N$. We then obtain an action of $\Gamma$ using a group morphism $\Gamma\to \Z$. Indeed, if $g$ is the genus of $S$, the group $\pi_1(S)$ is generated by $4g$ elements denoted by $a_i,b_i,a_i^{-1},b_i^{-1}$,
$1\leq i\leq g$, with the relation $a_1b_1a_1^{-1}b_1^{-1}\ldots a_gb_ga_g^{-1}b_g^{-1}=1$. So, we can send $a_1$ to the homeomorphism $f$ and the others $a_i$, $b_i$ to the identity.
Notice that in this case all positive $\ddbar$-closed currents directed by the lamination are closed, see e.g. 
\cite{FSW}.
\end{example}

\begin{theorem}\label{th_finite_entropy_regular}
Let $(X,\Lc)$ be a smooth   
compact lamination  by hyperbolic  Riemann surfaces.
Then, the entropy $h(\Lc)$ is finite. 
\end{theorem}
\proof
We will use the notations as in the proof of Theorem \ref{th_poincare_holder}. Consider $R$ large enough such that $e^{-R/3}\ll \epsilon$. 
We have seen in the proof of Theorem \ref{th_poincare_holder} that if $\dist_X(x,y)\leq e^{-10\kappa d^{-1}AR}$ then $x$ and $y$ are $(R/3,\epsilon)$-close. 
So, the maximal number of mutually $(R/3,\epsilon)$-separated points is smaller than a constant times $e^{10\kappa d^{-1}ANR}$ if the lamination is embedded in $\R^N$. So, the entropy $h(\Lc)$ is at most equal to 
$30\kappa d^{-1}AN$. 

Note that we can also apply here Proposition \ref{prop_abstract_finite_entropy} with $\varphi=e^{-t}$.  Indeed, using the arguments as above, we can show that $\dist_{R/3}\lesssim e^{10\kappa d^{-1} AR}\dist_X+\varphi(R/3)$. 
\endproof

Assume now that $X$ is a Riemannian manifold of dimension $k$ and $\Lc$ is transversally smooth. 
We can introduce various functionals  in order to describe  the dynamics.
For example we can introduce  dimensional entropies for a foliation   as done in Buzzi \cite{Buzzi} for maps.
For an interger $1\leq l\leq k$, consider the family $\Dc_l$ of manifolds
of dimension $l$ in $X$  which are smooth up to the boundary. Define 
$$h_l(\Lc):=\sup  \big \{ h(D):\  D\in \Dc_l\big\},$$
where $h(D)$ is the entropy restricted to the set $D$ as in the abstract setting.
Clearly, this sequence of entropies is increasing with $l$. 

We can also define 
$$\widetilde\chi_l(x):=\sup_{\epsilon>0}\sup_D\limsup_{R\to\infty}-{1\over R}\log\volume_l(D\cap B_R(x,\epsilon)).$$ 
Here, the supremum is taken over $D\in \Dc_l$ with $x\in D$. The 
Bowen ball  $B_R(x,\epsilon)$ is associated to the lamination and is defined as in the abstract setting. 
The volume $\volume_l$ denotes the Hausdorff measure of dimension $l$. 
The function $\widetilde\chi_l$ 
is the analog of the sum of $l$ largest Lyapounov exponents for dynamics of maps on manifolds. 
It measures how quickly the leaves get apart. 
We can consider this function relatively to a harmonic measure and show that it is constant when the measure is extremal.
The definitions of $h_l$ and $\widetilde \chi_l$ can be extended to the case of Riemannian foliations.

\begin{remark}\rm 
Assume that the lamination admits non-hyperbolic leaves and their union $Y$ is a closed subset. We can consider the entropy outside $Y$, but this quantity can be infinite. We can in this case modify the distance outside $Y$, e.g.  consider
$$\dist_X'(x,y)= \min\Big\{\dist_X(x,y),\inf_{x',y'\in Y}\dist_X(x,x')+\dist_X(y,y')\Big\}.$$
This means that we travel in $Y$ with zero cost.
The notion is natural because the Poincar\'e pseudo-distance vanishes on non-hyperbolic Riemann surfaces. 

For foliations on $\P^k$, we can also consider their pull-back using generically finite holomorphic maps from a projective manifold to $\P^k$ in order to get hyperbolic foliations.
\end{remark}

\section{Entropy of harmonic  measures}\label{section_entropy_metric}

We are going to discuss a notion of entropy for harmonic measures associated to laminations.
We first consider the abstract setting as in the beginning of Section \ref{section_entropy} for a family of distances 
$\{\dist_t\}_{t\geq 0}$ on a metric space $(X,\dist_X)$. Let $m$ be a probability measure on $X$. 
Fix positive constants $\epsilon,\delta$ and $t$.  Let $N_m(t,\epsilon,\delta)$ be the minimal
number of balls of radius $\epsilon$ relative to the metric
$\dist_t$ whose union has at least $m$-measure $1-\delta.$ 
{\it The entropy} of $m$ is    defined by the following formula
$$
h_\Dc(m):=\lim_{\delta\to 0}\lim_{\epsilon\to 0}\limsup_{t\to\infty}{1\over t} \log N_m(t,\epsilon,\delta). 
$$

We have the following general property.

\begin{proposition}
Let $(X,\dist_X)$ and $\Dc$ be  as  above. Then for any probability measure $m$  on $X$, we have 
$$h_\Dc(m)\leq  h_\Dc(\supp(m)).$$
\end{proposition}
\begin{proof}
Define $Y:=\supp(m)$. 
Choose    a  maximal family of $(t,\epsilon)$-separated points $x_i$ in $Y$. The family of $B_t(x_i,\epsilon)$ covers
$Y$. So, for every $\delta>0$ 
$$
N_m(t,\epsilon,\delta)\leq  M(Y,t,\epsilon).
$$ 
It follows from Proposition \ref{prop_comparison} that $h_\Dc(m)\leq  h_\Dc(Y).$
\end{proof}

As in Brin-Katok's theorem \cite{BrinKatok}, we can introduce the {\it local entropies} of $m$ at 
$x\in X$ by
$$h_\Dc^+(m,x,\epsilon):=\limsup_{t\rightarrow \infty} -{1\over t}\log m(B_t(x,\epsilon)),\qquad
h_\Dc^+(m,x):=\sup_{\epsilon> 0} h_\Dc^+(m,x,\epsilon), $$
and
$$h_\Dc^-(m,x,\epsilon):=\liminf_{t\rightarrow \infty} -{1\over t}\log m(B_t(x,\epsilon)),
\qquad
h_\Dc^-(m,x):=\sup_{\epsilon> 0} h_\Dc^-(m,x,\epsilon),
$$
where $B_t(x,\epsilon)$ denotes the ball centered at  $x$ of radius $\epsilon$ with respect to the distance $\dist_t$.

Note that
in the case of ergodic invariant measure associated with a continuous map on a metric compact space, the above notions of entropies coincide with the classical entropy of $m$, see Brin-Katok \cite{BrinKatok}.

Let $(X,\Lc)$ be  a Riemann  surface lamination such that its leaves are 
hyperbolic.
Since we do not assume that $X$ is compact, the discussion below can be applied to the regular part of a singular lamination.

The Poincar\'e metric $\omega_P$  provides a Laplacian $\Delta_P$ along the leaves. Recall that  a probability measure
$m$ is  {\it harmonic}  if it is  orthogonal  to  continuous functions  $\phi$ which  can be written $\phi=\Delta_P\psi$ where $\psi$ is a continuous function, smooth along the leaves and having compact support in $X$. In a flow box $\U=\B\times \T$ with $\B$ open set in $\C$, we can write
$$m=\int m_s d\mu(s),$$
where $\mu$ is a positive measure on $\T$, $m_s=h_s\omega_P$ is a measure on 
$\B\times\{s\}$ and $h_s$ is a positive harmonic function on this plaque. We refer to \cite{DinhNguyenSibony2, FornaessSibony1, Garnett} for more details and the relation with the notion of $\ddbar$-closed current.

Recall that in Section \ref{section_entropy}  we have associated to $(X,\Lc)$ a family of distances $\{\dist_t\}_{t\geq 0}$. Therefore, we can associate to $m$ a metric entropy and local entropies defined as above in the abstract setting. Recall that a harmonic probability measure $m$ is {\it extremal} if all harmonic probability measures $m_1,m_2$ satisfying $m_1+m_2=2m$ are equal to $m$. We have the following result.

\begin{theorem} \label{th_local_entropy}
Let $(X,\Lc)$ be a compact smooth  lamination by hyperbolic Riemann surfaces.
Let $m$ be  a  harmonic probability measure.  Then, the local entropies
$h^\pm$ of $m$ are constant on leaves.  
In particular, if $m$ is  extremal,  then  $h^\pm$  are constant $m$-almost everywhere.
\end{theorem}

In fact, the result holds for compact laminations which are not smooth, provided that 
$A^{-1}\dist_X\leq\dist_P\leq A\dist_X$ with the same constant $A>0$ for all plaques of a suitable atlas. We have seen that these inequalities hold when the lamination is smooth. 

Fix a covering of $X$ by a finite number of flow boxes
$\U=\B\times\T$, where $\B$ is the disc of center 0 and of radius 2 in $\C$ and $\T$ is a 
ball of center $s_0$ and of radius 2 in a complete metric space. 
We assume that the boxes $\U'=\D\times\T$ cover $X$. 
For simplicity, in what follows, we identify
the distance $\dist_X$ on $\U$ with the one induced by the distance on $\T$ and the Euclidean distance on $\B$. Denote by $\T_r$ the ball of center $s_0$ and of radius $r$ in $\T$.  
Fix also a constant $\delta>0$ such that if $\phi$ is a covering map of a leaf, then the image by $\phi$ of
any subset of Poincar\'e diameter $2\delta$  is contained in a flow box. 
  
The following lemma gives us a description of the intersection of Bowen balls with plaques. 

\begin{lemma} \label{lemma_bowen_slice}
Let $\epsilon>0$ be a fixed constant small enough. Then, 
there is a constant $A>0$ satisfying the following properties. Let $y$ and $y'$ be two points in $\D\times \{s\}$ with $s\in\T_1$ and $R>0$ be a constant. If $\dist_X(y,y')\leq A^{-1}e^{-R}$ then $y$ and $y'$ are $(R,\epsilon)$-close. If 
$\dist_X(y,y')\geq Ae^{-R}$ then $y$ and $y'$ are $(R,\epsilon)$-separated.
\end{lemma}
\proof 
Fix a constant $A>0$ large enough depending on $\epsilon$. We prove the first assertion. Assume that  $\dist_X(y,y')\leq A^{-1}e^{-R}$.
Let $L$ denote the leaf containing $\B\times\{s\}$. 
Let $\phi'$ be a covering map of $L$ such that $\phi'(0)=y'$. So, there is a point $a\in \D$ such that $\phi'(a)=y$ and $\dist_P(0,a)\ll e^{-R}$. By Lemma \ref{lemma_poincare_separated}, there is an automorphism $\tau$ of $\D$, close to the identity on $\D_R$, such that $\tau(0)=a$. Define $\phi:=\phi'\circ\tau$. This is also a covering map of $L$. It is clear that $\dist_{\D_R}(\phi,\phi')\leq \epsilon$. Therefore, $y$ and $y'$ are $(R,\epsilon)$-close. 

For the second assertion, assume that $\dist_X(y,y')\geq Ae^{-R}$ but $y$ and $y'$ are $(R,\epsilon)$-close.
By definition of Bowen ball, we can find two covering maps $\phi,\phi':\D\to L$ such that $\phi(0)=y$, $\phi'(0)=y'$ and $\dist_{\D_R}(\phi,\phi')\leq\epsilon$. In particular, we have $\dist_X(y,y')\leq\epsilon$. 
We can find a point $a\in\D$ such that $\phi'(a)=y$ and  $\dist_P(0,a)=\dist_P(y,y')$. Since $\epsilon$ is small and $A$ is large, we have 
$$e^{-R}\ll \dist_X(y,y')\lesssim \dist_P(0,a)\leq \delta.$$ 

There is also an automorphism $\tau$ of $\D$ such that $\tau(0)=a$ and $\phi=\phi'\circ\tau$. 
The last assertion in Lemma \ref{lemma_poincare_separated} implies by continuity that we can find a point $z\in \D_R$ satisfying
$\epsilon\ll \dist_P(z,\tau(z))<\delta$. Finally, the property of $\delta$ implies that $\dist_X(\phi(z),\phi'(z))>\epsilon$. This is a contradiction.
\endproof

We now introduce a notion of transversal entropy which can be extended to a general lamination. In what follows, if $V$ is a subset of $\U$, we denote by $\widetilde V$ its projection on $\T$. The measure $m$ can be written in a unique way on $\U$ as
$$m=\int m_s d\mu(s),$$
where $m_s=h_s\omega_P$ is as above with the extra condition $h_s(0)=1$. 

By Harnack's principle, the family of positive harmonic functions $h_s$ is locally uniformly bounded from above and from below  by strictly positive constants. This implies that the following notions of transversal entropy do not depend on the choice of flow box. Define 
$$\widetilde h ^+(x):=\sup_{\epsilon>0} \limsup_{R\to\infty} -{1\over R}\log \mu(\widetilde B_R(x,\epsilon))$$
and
$$\widetilde h ^-(x):=\sup_{\epsilon>0} \liminf_{R\to\infty} -{1\over R}\log \mu(\widetilde B_R(x,\epsilon)).$$

Note that we can also use $\widetilde B_R(x,\epsilon)$ in order to define a notion of topological entropy on $\T$. 

\begin{lemma} \label{lemma_slice_entropy}
We have $h^\pm=\widetilde h ^\pm+2$. 
\end{lemma}
\proof
We can assume that $x$ belongs to $\D\times\T_1$. 
By Lemma \ref{lemma_bowen_slice}, the intersection of $B_R(x,\epsilon)$ with a plaque is of diameter at most equal to $2Ae^{-R}$. Since $h_s$ is bounded from above uniformly on $s$, we deduce that 
$$m(B_R(x,\epsilon))\lesssim e^{-2R}\mu(\widetilde B_R(x,\epsilon)).$$
It follows that $h^\pm\geq \widetilde h ^\pm+2$. 

We apply the first assertion in Lemma \ref{lemma_bowen_slice} to $\epsilon/2$ instead of $\epsilon$. We deduce that if a plaque $\D\times\{s\}$ intersects $B_R(x,\epsilon/2)$ then its intersection with $B_R(x,\epsilon)$ contains a disc of radius $A^{-1}e^{-R}$. It follows that 
$$m(B_R(x,\epsilon))\gtrsim e^{-2R}\mu(\widetilde B_R(x,\epsilon/2)).$$
This implies that $h^\pm\leq \widetilde h ^\pm+2$ and completes the proof of the lemma.
\endproof

\noindent
{\bf End of the proof of Theorem \ref{th_local_entropy}.} 
Let   $x$ and $y$  be in the same leaf $L$.  We want to prove that $h^\pm(x)=h^\pm(y)$. It is enough to consider the case where $x$ and $y$ are close enough and to show that $h^\pm(x)\leq h^\pm(y)$. So, using the same notation as above, we can assume that $x$ and $y$ belong to $\D\times\{s_0\}$. 
We show that $\widetilde h ^\pm(x)\leq \widetilde h ^\pm(y)$.
Fix a constant $\epsilon>0$ small enough and a constant $\gamma>0$ large enough.  It suffices to show for large $R$ that
$$\widetilde B_R(y,\epsilon)\subset \widetilde B_{R-\gamma}(x,\epsilon).$$

Let $y'$ be a point in the intersection of $B_R(y,\epsilon)$ with a plaque $\B\times \{s_0'\}$. We have to show that $B_{R-\gamma}(x,\epsilon)$ intersects also $\B\times \{s_0'\}$. Let $L$ and $L'$ denote the leaves containing 
$\B\times \{s_0\}$ and $\B\times \{s_0'\}$ respectively.  
Consider universal maps $\phi:\D\to L$ and $\phi':\D\to L'$ such that $\phi(0)=y$, $\phi'(0)=y'$ and $\dist_{\D_R}(\phi,\phi')\leq\epsilon$. Let $a\in\D$ such that $\phi(a)=x$. Since $x$ is close to $y$, we can find $a$ close to $0$. Let $\tau$ be an automorphism of $\D$ such that $\tau(0)=a$. Since $a$ is close to 0 and $R$ is large, the image of $\D_{R-\gamma}$ by $\tau$, i.e. the disc of center $a$ and of radius $R-\gamma$, is contained in $\D_R$. We deduce that 
$\dist_{\D_{R-\gamma}}(\phi\circ \tau,\phi'\circ\tau)\leq\epsilon$. In particular, $\phi'(\tau(a))$ is a point in $B_{R-\gamma}(x,\epsilon)$. This implies the result.
\hfill $\square$

\bigskip

Let $(X,\Lc)$ be as in Theorem \ref{th_local_entropy}.
Let $m$ be an extremal harmonic probability measure. For simplicity, we will denote by $h^\pm(m)$ the constants associated with the local entropy functions $h^\pm$. We have the following result.

\begin{proposition}
With the above notation, we have
$$ h^-(m)\leq  h(m)\leq  h^+(m)\leq  h(\Lc). $$  
In particular, $h(\Lc)$ is always larger or equal to $2$.
\end{proposition}
\begin{proof}
We will use the notations $h^\pm(x,\epsilon)$, $B_t(x,\epsilon)$, $N_m(t,\epsilon,\delta)$ and $N(X,t,\epsilon)$ as in the abstract setting. First,  we will prove that  $h(m)\leq  h^+(m)$. 

Fix constants $\alpha>0$ and $0<\delta<1/4$. 
Given a constant $\epsilon>0$ small enough, we can find  a  subset  $X'\subset X$ with
   $m(X')\geq  1-\delta$ such that for $t$ large enough and for $x\in X'$,  we have
$${1\over t}\log {1\over m(B_t(x,\epsilon))}\leq  h^+(x,\epsilon)+{\alpha \over 2}\leq h^+(m)+\alpha. $$
So, for such $\epsilon$, $x$ and $t$, we have   
$$m(B_t(x,\epsilon))\geq  e^{-t(h^+(m)+\alpha)}.$$

Consider a maximal family of disjoint balls
$B_t(x_i,\epsilon)$ with center $x_i\in X'$. The union of $B_t(x_i,2\epsilon)$ covers $X'$ which is of measure at least $1-\delta$. Therefore, we have
$$ 1\geq   m \big  (\bigcup B_t(x_i,\epsilon) \big)\geq   N_m(t,2\epsilon,\delta)e^{-t(h^+(m) +\alpha)}.$$
It follows that  $ N_m(t,2\epsilon,\delta)\leq e^{t(h^+(m) +\alpha)}$ for $t$ large enough. 
Since this inequality holds for every $\alpha>0$, we deduce that $h(m)\leq  h^+(m)$.

We now prove that $h^-(m)\leq h(m)$. As above, given $\epsilon>0$, we can find a subset $X''$ with  $m(X'')\geq  3/4$ such that for $t$ large enough and for $x\in X''$,  we have
$$m(B_t(x,6\epsilon))\leq  e^{-t(h^-(m)-\alpha)}.$$
Consider a minimal family of balls $B_t(x_i,\epsilon)$ which covers a set of measure at least $3/4$. 
By removing the balls which do not intersect $X''$, we still have a family which covers a set of measure at least $1/2$. So, each ball $B_t(x_i,\epsilon)$ is contained in a ball $B_t(x_i',2\epsilon)$ centered at a point $x_i'\in X''$. 
Vitali's covering lemma implies  the  existence of a  finite sub-family of disjoint balls  $B_t(y_j,2\epsilon)$
such that $\bigcup B_t(y_j,6\epsilon)$  covers $\bigcup B_t(x'_i,2\epsilon)$.
Hence,
$$1/2\leq   m \big  (\bigcup B_t(y_j,6\epsilon) \big)\leq   N_m(t,\epsilon,3/4)e^{-t(h^-(m) -\alpha)}.$$
It follows that  $2N_m(t,\epsilon,3/4)\geq e^{t(h^-(m)- \alpha)}$ and therefore $h(m)\geq  h^-(m)$.

It remains  to show that  $h^+(m)\leq  h(\Lc)$. Suppose  in order to get a contradiction that
$h(\Lc)\leq  h^+(m)-3\delta$  for some $\delta>0.$ For any  $\epsilon>0$, there exists
$t_0$ large  enough   such that for all $t\geq t_0$
$${1\over t} \log N(X,t,\epsilon)\leq  h(\Lc)+{\delta\over 2}\cdot$$
In particular, we have
$$N(X,t,\epsilon)\leq  {1\over 4t^2} e^{(h(\Lc)+\delta)t}.$$

Fix now an $\epsilon>0$ small enough and then $t_0$ large enough.
Since  $h^+(m)-\delta\geq  h(\Lc)+2\delta,$ we have $m(\Lambda)>1/2$ where
$$\Lambda:=\Big\lbrace x\in X:\  \sup_{t\geq t_0}  -{1\over t} \log m(B_t(x,2\epsilon))\geq  h(\Lc)+2\delta  
\Big\rbrace.$$
Define
$$\Lambda_t:=\Big\lbrace  x\in X:\   -{1\over t} \log m(B_t(x,2\epsilon))\geq  h(\Lc)+\delta  \Big\rbrace$$
and
$$\Lambda'_t:=\Big\lbrace  x\in X:\   -{1\over t} \log m(B_t(x,2\epsilon))\geq  h(\Lc)+2\delta  \Big\rbrace$$
 
Since $t$ is large, we have $\Lambda_t'\subset \Lambda_{t+\alpha}$ for $0\leq\alpha\leq 1$.
Consider integer numbers $n$ larger than $t_0$. We have
 $$\Lambda=\bigcup_{t\geq t_0} \Lambda'_t\subset\bigcup_{n\geq t_0} \Lambda_n.$$ So, we can find
  $n\geq t_0$  such that $m(\Lambda_n)>1/(4n^2)$.
  Hence,  by definition of $\Lambda_t$, we get
  $$ N(\Lambda_n,n,2\epsilon)>  {1\over 4n^2} e^{(h(\Lc)+\delta)n}.$$
Therefore, 
$$N(X,n,\epsilon)> {1\over 4n^2} e^{(h(\Lc)+\delta)n}.$$
This  is  a   contradiction.
\end{proof}

Here are some fundamental problems concerning metric entropies for Riemann surface laminations. Assume here that $(X,\Lc)$ is a compact smooth lamination by hyperbolic Riemann surfaces but the problems can be stated in a more general setting. 

\begin{problem}\rm
Consider extremal harmonic probability measures $m$.  Is the following {\it variational principle} always true
$$h(\Lc)=\sup_m h(m)\quad  ?$$
\end{problem}

Even when this principle does not hold, it is interesting to consider the invariant 
$$h(\Lc)-\sup_m h(m)$$
and to clarify the role of the hyperbolic time in this number.

\begin{problem}\rm
If $m$ is above, is the identity $h^+(m)=h^-(m)$ always true ?
\end{problem}

We believe that the answer is affirmative and gives an analog of the Brin-Katok theorem.

Notice that there is a notion of entropy for harmonic measures introduced by Kaimanovich \cite{Kaimanovich}.
Consider a metric $\omega$ of bounded geometry on the leaves of the lamination. Then,  
we can consider the heat kernel $p(t,\cdot,\cdot)$ associated to the Laplacian
determined by this metric.  If $m$ is  a  harmonic probability measure on $X,$    Kaimanovich defines the entropy
of $m$ as
$$
h_K(m):=\int  dm(x)\Big( \lim\limits_{t\to\infty} -{1\over t}\int p(t,x,y)\log p(t,x,y)\omega(y)\Big).
$$ 
He  shows that the limit  exists and is  constant $m$-almost everywhere when $m$ is  extremal. 

This notion of entropy has  been extensively  studied for universal covering of a compact Riemannian manifold, see e.g. Ledrappier \cite{Ledrappier}. It does not seem that  it was   studied for compact foliations  with singularities.
So, it would be of interest to find relations with our notions of entropy  defined above.
In  Kaimanovich's entropy,  the  transverse  spreading is  present through the variation of the heat kernel from leaf to leaf. It would be also interesting to make this dependence more explicit.

\noindent
T.-C. Dinh, UPMC Univ Paris 06, UMR 7586, Institut de
Math{\'e}matiques de Jussieu, 4 place Jussieu, F-75005 Paris,
France.\\
{\tt  dinh@math.jussieu.fr}, {\tt http://www.math.jussieu.fr/$\sim$dinh}

\medskip

\noindent
V.-A.  Nguy{\^e}n, Math{\'e}matique-B{\^a}timent 425, UMR 8628, Universit{\'e} Paris-Sud,
91405 Orsay, France.\\
 {\tt VietAnh.Nguyen@math.u-psud.fr}, {\tt http://www.math.u-psud.fr/$\sim$vietanh}

\medskip

\noindent
N. Sibony, Math{\'e}matique-B{\^a}timent 425, UMR 8628, Universit{\'e} Paris-Sud,
91405 Orsay, France.\\
{\tt Nessim.Sibony@math.u-psud.fr}

\end{document}